\documentclass[10pt,a4paper,reqno]{amsart}

  \usepackage{amsthm}
  
  \usepackage{amsmath}
  \usepackage{amssymb}
  \usepackage{mathtools}
\allowdisplaybreaks

    \newtheorem{thm}{Theorem}
      \newtheorem{prop}[thm]{Proposition}
      \newtheorem{conj}[thm]{Conjecture}
      \newtheorem{formula}[thm]{Formula}
      \newtheorem{cor}[thm]{Corollary}
      \newtheorem{lemma}[thm]{Lemma}
      \newtheorem{observation}[thm]{Empirical observation}
     
  \theoremstyle{definition}

\usepackage{newpxmath}
     \usepackage{newpxtext} 

\usepackage[hyphens]{url}
\usepackage{enumerate}
\usepackage[breaklinks=true]{hyperref}
  \usepackage{microtype}

\newcommand{\dd}{\mathrm{d}}
\newcommand{\vect}[1]{\vec{#1}}
\newcommand{\comp}[1]{\overline{#1}}
\newcommand{\dr}[1]{\frac{\mathrm{d}^{#1}}{\mathrm{d}r^{#1}}}

\newcommand{\define}[1]{\textbf{#1}}
\newcommand{\R}{\mathbb{R}}

\newcommand{\N}{\mathbb{N}}

\DeclareMathOperator{\num}{numerator}
\DeclareMathOperator{\den}{denominator}
\title[On the magnitude of odd balls via potential functions]
  {On the magnitude of odd balls\\via potential functions}
\author{Simon Willerton}

\begin{document}

\begin{abstract}
Magnitude is a measure of size defined for certain classes of metric spaces; it arose from ideas in category theory.  In particular, magnitude is defined for compact subsets of Euclidean space and 
Barcel\'o and Carbery gave a procedure for calculating the magnitude of balls in odd dimensional Euclidean spaces.  In this paper their approach is modified in various ways: this leads to an explicit determinantal formula for the magnitude of odd balls and leads to the conjecturing of a simpler formula in terms of Hankel determinants.  This latter formula is proved using a rather different approach in ``The magnitude of odd balls via Hankel determinants of reverse Bessel polynomials'', but the current paper provides the reasoning that led to the formula being conjectured.   Finally, an empirically-tested Hankel determinant formula for the derivative of the magnitude is conjectured.
\end{abstract}
\maketitle
\setcounter{tocdepth}{1}

\section{Introduction}
Throughout this paper we will adopt the convention that $n$ is an odd integer and $n=2p+1$, there will be reminders of this now and again.

\subsection{Background and overview}
Magnitude was introduced by Leinster~\cite{Leinster:Magnitude} as a measure of size for finite metric spaces; this was done by generalizing a notion of Euler characteristic for finite categories.  It soon became clear that the realm of definition of magnitude could be extended to a large class of \emph{infinite} metric spaces, this class includes compact subspaces of Euclidean spaces.  Mark Meckes~\cite{Meckes:PositiveDefinite, Meckes:MagnitudeDimensions} gave various equivalent ways of defining the magnitude on such spaces, one of these ways using a notion of potential function for compact subsets of Euclidean space, another using a notion of weight distribution.

Magnitude was known to be connected with many classical concepts including volume, total scalar curvature~\cite{Willerton:SpheresSurfaces}, and Minkowski dimension~\cite{Meckes:MagnitudeDimensions} and was conjectured to be connected with intrinsic volumes of convex sets~\cite{LeinsterWillerton:AsymptoticMagnitude}.  However, the precise magnitude of any compact set with dimension greater than one was unknown.  Then, utilizing the spherical symmetry, Barcel\'o and Carbery~\cite{BarceloCarbery} were able to give an algorithm for calculating the potential function of any odd dimensional ball of given radius, and from this they could give a procedure for calculating the ball's magnitude.  In general, for a fixed, odd dimension, this process gives the magnitude as a rational function in the radius of the ball.  They were able to compute this function essentially by hand in dimensions $1$, $3$, $5$ and $7$.  Here are the formulas they found.
\begin{align*}
\left|B^ 1 _R\right|
&=
R+1
\\
\left|B^ 3 _R\right|
&=
\frac{R^{3} + 6R^{2} + 12  R + 6}{3!}
\\
\left|B^ 5 _R\right|
&=
\frac{R^{6} + 18 \, R^{5} + 135 \, R^{4} + 525 \, R^{3} + 1080 \, R^{2} + 1080 \, R + 360}{5! \, {\left(R + 3\right)}}
\\
\left|B^ 7 _R\right|
&=
\frac{  R^{10} +
40  R^{9} + 720  R^{8} +  \dots 
+ 1814400  R^{2} + 1209600  R + 302400
}{7!\, {\left(R^{3} + 12  R^{2} + 48  R + 60\right)}}
\end{align*}

In work~\cite{Willerton:Hankel}  inspired by their paper, but logically independent, I gave an explicit formula for $|B_R^n| $ as a ratio of Hankel determinants of reverse Bessel polynomials (see below for what these terms mean); I prove this formula by using the weight distribution approach to magnitude, rather than the potential function approach.  The numerator and denominator in the formula are given combinatorial interpretations as path counting polynomials which means that various properties (which are probably evident in the examples listed above) such as positivity of the coefficients and bounds on degrees are straightforward to obtain.  However, in~\cite{Willerton:Hankel} no explanation is given for how the Hankel determinant formula was arrived at; it is just pulled out of the air and proved to be correct.

One purpose of the current paper is to explain how that formula was conjectured and to provide a bridge from~\cite{BarceloCarbery} to~\cite{Willerton:Hankel}.  Another purpose is show an alternative approach to calculating the magnitude which could be useful.  It is quite possible that the approach given here could be used to prove the Hankel determinant formula, although I have been unable to do this.

\subsection{Two sequences of functions}
\label{sec:Functions}
There are two sequences of functions that will be used throughout the paper.  First we have the sequence of function $(\psi_i\colon (0,\infty)\to \R)_{i=0}^\infty$, we will define this sequence inductively via
\[
\psi_0(r)\coloneqq e^{-r}
\quad
\text{and}
\quad 
\psi_{i+1}(r)\coloneqq -\tfrac{1}{r}\psi_{i}'(r).
\]
We will deduce many properties of these from this definition in Section~\ref{sec:Functions}.  
It is immediate by induction that $\psi_i(r)\in e^{-r}\N[\tfrac{1}{r}]$. In fact this is how we will define the second sequence $(\chi_i)_{i=0}^\infty$, which is a sequence of polynomials called the \define{reverse Bessel polynomials}: $\chi_i(r):=e^{r}r^{2i}\psi_i(r)$.
The first few of these functions are as follows:
\begin{align*}
\psi_0(r) & = e^{-r};
&
\chi_0(R)&=1;\\
\psi_1(r) & = e^{-r}\bigl(
\tfrac{1}{r}
\bigr);
&
\chi_1(R)&=R;
\\
\psi_2(r) & = e^{-r}\bigl(
\tfrac{1}{r^{2}} + \tfrac{1}{r^{3}}
\bigr);
&
\chi_2(R)&=R^2+R;
\\
\psi_3(r) & = e^{-r}\bigl(
\tfrac{1}{r^{3}} + \tfrac{3}{r^{4}} + \tfrac{3}{r^{5}}
\bigr);
&
\chi_3(R)&=R^3+3R^2+3R;
\\
\psi_4(r) & = e^{-r}\bigl(
\tfrac{1}{r^{4}} + \tfrac{6}{r^{5}} + \tfrac{15}{r^{6}} + \tfrac{15}{r^{7}}
\bigr);
&
\chi_4(R)&=R^{4} + 6 R^{3} + 15 R^{2} + 15 R.
\end{align*}
\subsection{Further background}
Marks Meckes~\cite{Meckes:MagnitudeDimensions} showed that the magnitude of a compact subset $X$ of $\R^n$ can be determined via a potential function.  In this context a potential function means a $p$-times differentiable function $h\colon \R^n\to \R$ such that 
\begin{itemize}
\item $h=1$ on $X$;
\item $(1-\Delta)^{p+1} h=0$ on $\R\setminus X$.
\end{itemize}
The magnitude of $X$ can then be calculated via
\[
|X|=
\frac{1}{n!\,\omega_n}\int_{\mathbb{R}^n} (1+\left \| x\right\|^2)^{p+1} \bigl| \widehat{h}(x)\bigr |^2\,\mathrm{d}x,
\]
where $\widehat{h}$ is the Fourier transform of $h$.

Barcel\'o and Carbery showed, provided the boundary of $X$ was sufficiently smooth, that the formula could be written in terms of the volume of $X$ and an integral over the boundary of $X$:
\begin{equation}
\label{eq:BCmag}
|X|=\frac{1}{n!\,\omega_n}\biggl(\text{vol}(X)+\sum_{(p+1)/2<j\le p+1}(-1)^j\binom{p+1}{j}\int_{\partial X} \frac{\partial}{\partial \nu} \Delta^{j-1}h\,\mathrm{d}{s}\biggr),
\end{equation}
where $\frac{\partial}{\partial \nu}$ means the normal derivative at the boundary and $\Delta$ is the Laplacian operator, $\Delta f=\sum_{i=1}^n\frac{\partial^2}{\partial x_i^2}f$.

Leinster and Meckes~\cite{LeinsterMeckes:Survey} later showed that, provided that the potential function was integrable, the magnitude can be expressed simply as an integral of the potential function:
\begin{equation}
\label{eq:LMmag}
\left|X\right| = \frac{1}{n!\,\omega_n}\int_{\R^n}h(x)\, \mathrm{d}x .
\end{equation}

Moving specifically now to the case that $X$ is the $n$-ball, Barcel\'o and Carbery~\cite{BarceloCarbery} use the fact that the potential function $h$ will be spherically symmetric and find the potential function for the $n$-ball by finding all spherically symmetric solutions of the equation $(I-\Delta)^{p+1} g=0$ on $\R^{2p+1}\backslash \{0\}$ with appropriate decay at infinity; the set of solutions is precisely the set of linear combinations of the functions $\psi_0(r),\dots, \psi_p(r)$ which are defined above, with $r$ being the radial coordinate.   This means that the potential function $h$ on the $n$-ball is of the form 
\[
h(r)=\begin{cases}
1&r<R\\
\sum_{i=0}^p \alpha_i\psi_i(r)&r\ge R,
\end{cases}
\]
for some set of coefficients $\{\alpha_i\}_{i=0}^{p}$ which depend on $R$.
They then set up boundary conditions for the differential equation in the following way which is apparently natural for analysts.  They first define the set of differential operators $\{\mathcal{D}^i\}_i=0^\infty$ at the boundary $\partial X$ in terms of powers of the Laplacian $\Delta$ by
$\mathcal{D}^{2j}:=\Delta^j$ and $D^{2j+1}:=\frac{\partial}{\partial\nu}\Delta^j$.  Then the boundary conditions are
\begin{equation}
\label{eq:BCBC}
h(R)=1; 
\quad
\mathcal{D}h(R)=0; 
\quad
\mathcal{D}^2h(R)=0;
\quad
\dots;
\quad
\mathcal{D}^ph(R)=0.
\end{equation}
%
This leads to the following linear system for the coefficients, where  $p$ is assumed even --- the odd case involves removing the bottom row --- and where for reasons of space $\psi_i$ is written for $\psi_i(R)$.
\[
\left(
\begin{smallmatrix}
\psi_0&  \psi_1& \dots&&&\dots&\psi_{p-2}&\psi_{p-1}&\psi_{p}\\
\psi_1& \psi_2&  \dots&&&\dots&\psi_{p-1}&\psi_{p}&\psi_{p+1}\\
2p\psi_1& 2(p-1)\psi_2& \dots&&&\dots&2\cdot 2\psi_{p-1}&2\psi_{p}&0\\
2p\psi_2& 2(p-1)\psi_3& \dots&&&\dots&2\cdot 2\psi_{p}&2\psi_{p+1}&0\\
4p(p-1)\psi_2& 4(p-1)(p-2)\psi_3& \dots&&&\dots&4\cdot 2\cdot 1\psi_{p}&0&0\\
4p(p-1)\psi_3& 4(p-1)(p-2)\psi_4& \dots&&&\dots&4\cdot 2\cdot 1\psi_{p+1}&0&0\\
\vdots&&&&&&&&\vdots\\
\tfrac{2^{\frac{p}{2}}p!}{(p/2)!}\psi_{p/2}
&\tfrac{2^{\frac{p}{2}}(p-1)!}{(p/2-1)!}\psi_{p/2+1}
&\dots&
2^{\frac{p}{2}}(p/2)!\psi_{p}
&0&\dots
&0&0&0
\end{smallmatrix}
\right)
\footnotesize
\begin{pmatrix}
\alpha_0\\
\alpha_1\\
\alpha_2\\
\alpha_3\\
\alpha_4\\
\alpha_5\\
\vdots\\
\alpha_{p}
\end{pmatrix}
=
\begin{pmatrix}
1\\
0\\
1\\
0\\
1\\
0\\
\vdots\\
1
\end{pmatrix}
\] 
For a given $n$, one can then solve this system to find $\alpha_0,\dots,\alpha_p$.  Barcel\'o and Carbery go on to describe a recursive algorithm for using this solution to obtain $\frac{\partial}{\partial \nu} \Delta^{j-1}h(R)$ for $(p+1)/2< j\le p+1$ and hence obtain the magnitude $|B^n_R|$ via formula~\eqref{eq:BCmag}.  This is how they calculated the formulae on the first page.

\subsection{What is in this paper}
In this paper we will do two things differently: we will use a different formulation of the boundary conditions and a different formula, namely~\eqref{eq:LMmag}, for the magnitude.  These lead us to a formula in terms of determinants and thence to the conjecture which is proved in~\cite{Willerton:Hankel}.

We take more naive boundary conditions.  We know that all derivatives of the potential function $h$ up to degree $p$ vanish at the boundary of the ball, so we write that as the vanishing of the higher normal derivatives.  In this spherically symmetric situation, the normal derivative $\frac{\partial}{\partial\nu}$ is just the radial derivative  $\frac{\mathrm{d}}{\mathrm{d}r}$ so we have the following boundary conditions:
\begin{equation}
\label{eq:NaiveBC}
h(R)=1; 
\quad
h'(R)=0; 
\quad
h''(R)=0;
\quad
\dots;
\quad
h^{(p)}(R)=0.
\end{equation}
We will see in Section~\ref{sec:FindingThePotential} that boundary conditions lead to the following linear system for the coefficients of $h$.
\begin{equation}
\label{eq:IntroHankelSystem}
\footnotesize
\begin{pmatrix}
\psi_0(R)&  \psi_1(R)& \psi_2(R)& \psi_3(R)& \dots&&\dots&\psi_{p}(R)\\
\psi_1(R)& \psi_2(R)& \psi_3(R)& \dots&&&\dots&\psi_{p+1}(R)\\
\psi_2(R)& \psi_3(R)& \dots&&&&\dots&\psi_{p+2}(R)\\
\psi_3(R)& \dots&&&&&\dots&\psi_{p+3}(R)\\
\vdots&&&&&&&\vdots\\
\\
\vdots&&&&&&\dots&\psi_{2p-1}(R)\\
\psi_{p}(R)&\dots&&&&&\dots&\psi_{2p}(R)
\end{pmatrix}
\begin{pmatrix}
\alpha_0\\
\alpha_1\\
\alpha_2\\
\alpha_3\\
\vdots\\
\\
\vdots\\
\alpha_{p}
\end{pmatrix}
=
\begin{pmatrix}
1\\
0\\
0\\
0\\
\vdots\\
\\
\vdots\\
0
\end{pmatrix}
\end{equation}
This is a `Hankel system' as the anti-diagonals are constant and it is evidently more symmetric than the Barcel\'o-Carbery system, it also has only one non-trivial entry on the right hand side, so can being viewed as being `simpler' than their system.  On the other hand the Barcel\'o-Carbery system has many zeros in the matrix and only uses $\psi_0,\dots, \psi_{p+1}$ whereas our matrix uses $\psi_0,\dots,\psi_{2p}$.  In any case our matrix can be reduced to theirs using elementary row operations together with the recurrence relation $R^2\psi_{i+2}(R)=\psi_i(R)+(2i+1)\psi_{i+1}(R)$.  The reduction and the proof of the  recurrence relation are left as exercises for the interested reader.

Note that as $e^{R}R^{2i}\psi_i(R)$ is a polynomial, namely the reverse Bessel polynomial $\chi_i(R)$, we can rewrite this as a matrix of polynomials  by scaling appropriately, this means writing $\tilde \alpha_i:=e^{-R}R^{2i}\alpha_i$, and $h(r)=\sum\tilde\alpha_i e^{R-r}(R/r)^{2i} \chi_i(r)$ for $r\ge R$.

We then use Leinster and Meckes formula~\eqref{eq:LMmag} for the magnitude from the potential function and we obtain the magnitude $|B_R^n|$ as a linear combination of the coefficients $\tilde\alpha_0,\dots,\tilde\alpha_p$.
Expressing the magnitude in this way allows it to be thought of as a solution of a linear system.  Using Cramer's Rule leads to Theorem~\ref{thm:MagDet} which gives the following explicit form of the magnitude, where $\xi_{p,0}(R)\dots\xi_{p,p}(R)$  are certain specific integer polynomials. 
   \[
    \left|B^{n}_R\right|
    =
    \frac{(-1)^p}{n!\,R}
    \frac{ \left|
    \begin{matrix}
      \chi_1(R)&  \dots&\chi_{p+1}(R)\\
      \vdots&&\vdots\\
      \chi_{p}(R)&\dots&\chi_{2p}(R) \\
      \xi_{p,0}(R)& \dots& \xi_{p,p}(R)   
    \end{matrix}
    \right|}
    {\left|
    \begin{matrix}
      \chi_0(R)&\dots&\chi_{p}(R)\\
      \vdots&&\vdots\\
      \chi_{p}(R)&\dots&\chi_{2p}(R)
    \end{matrix}
    \right|
    }
  \]

An empirical observation, which was made independently also by Barcel\'o and Carbery, is that the coefficient $\alpha_0$ of the potential function for the $(n+2)$-ball has the same numerator as the \emph{magnitude} of the $n$-ball.  Using the linear system for the $\tilde\alpha_0,\dots, \tilde\alpha_p$ and Cramer's Rule it is possible to calculate the numerator as proportional to the determinant $ [\chi_{i+j+2}(R)]_{i,j=0}^{p}$, this is the determinant of a Hankel matrix so is called a Hankel determinant.  The denominator we get for $|B^n_R|$ is proportional to $\det[\chi_{i+j}(R)]_{i,j=0}^{p}$ and a look at small values of $n$ leads to conjecturing Formula~\ref{conj:HankelFormula}, which is the following
 \[
   |B_R^n|
   \stackrel{?}{=}
   \frac{1}{n!\,R}
   \frac{\det\left( [\chi_{i+j+2}(R)]_{i,j=0}^{p}\right)}
       {\det\left([\chi_{i+j}(R)]_{i,j=0}^{p}\right)}.
\]
Whilst it is easy to check, for instance with SageMath, that Theorem~\ref{thm:MagDet} and Formula~\ref{conj:HankelFormula} give the same answer for $n<40$, I have been unable to prove that the expression in Theorem~\ref{thm:MagDet} is equal to that in Formula~\ref{conj:HankelFormula}.  However, using rather different methods, I have proved Formula~\ref{conj:HankelFormula} in~\cite{Willerton:Hankel}, where it is part of the main theorem.  Those methods, unfortunately, do not give any insight into why such a beautifully symmetric expression for the magnitude exists.

The paper ends with a justification of the following conjecture which relates the derivative of the magnitude function to the first not trivial derivative of the potential function at the boundary:
\[
  \frac{\dd}{\dd R} |B_R^n| \stackrel{?}{=}
 \frac{R^{n-1}}{(n-1)!} \left[ h^{(p+1)}(R)\right]^2.
\]
This is shown to be equivalent to the following explicit formula for the derivative of the magnitude function:
\[
\frac{\dd}{\dd R} |B_R^n| \stackrel{?}{=}
\frac{\bigl(\det
[\chi_{i+j+1}(R)]_{i,j=0}^{p}\bigr)^2
}
{(2p)!\,R^{2}\bigl(\det
[\chi_{i+j}(R)]_{i,j=0}^{p}\bigr)^2
}
.
\]

\section{Solving $(I-\Delta)^{p+1}g(r)=0$ on $\R^{2p+1}\backslash \{0\}$.}
The purpose of this section is to give a more streamlined proof of the result of Barcel\'o and Carbery~\cite{BarceloCarbery} that the functions $\psi_0,\dots,\psi_p$ span the space of asymptotically decaying, spherically symmetric functions on $R^{2p+1}\setminus\{0\}$ which are solutions of the differential equation $(I-\Delta)^{p+1}g(r)=0$.

We begin by proving basic properties of the sequence of functions $(\psi_i\colon \R\setminus\{0\}\to \R)_{i=0}^\infty$ from the definition given in Section~\ref{sec:Functions}.
\begin{thm}
\label{thm:psiDiffEqn}
For $i\ge 0$ we have
$\psi_i''(r)+\tfrac{2i}{r}\psi_i'(r)- \psi_i(r)=0$.
\end{thm}
\begin{proof}
This is proved by induction.  It is clearly true for $i=0$.  To prove the inductive step, begin by substituting in the inductive definition of $\psi_{i+1}$.
\begin{align*}
\psi_{i+1}''+\tfrac{2i+2}{r}\psi_{i+1}' &=
(-\tfrac{1}{r}\psi_{i}')''+\tfrac{2i+2}{r}(-\tfrac{1}{r}\psi_{i}')'\\
&= -(\tfrac{2}{r^3} \psi_{i}' -\tfrac{2}{r^2} \psi_i'' +\tfrac{1}{r} \psi_i''')  -\tfrac{2i+2}{r}(-\tfrac{1}{r^2}\psi_i' + \tfrac{1}{r} \psi_i'') \\
&= -\tfrac{1}{r}[\psi_i''' +\tfrac{2i}{r}\psi_i'' -\tfrac{2i}{r^2}\psi_i']\\
&= -\tfrac{1}{r} [\psi_i''+\tfrac{2i}{r}\psi_i' ]' =-\tfrac{1}{r}\psi_i' =\psi_{i+1}.
\end{align*}
The next to last equality comes from the inductive hypothesis.
\end{proof}
We can now get to the main property of the sequence $(\psi_i)$ that we are interested in.  We will switch perspectives, fix a positive integer $n$ and consider each $\psi_i$ as a spherically symmetric function on $\R^n\backslash\{0\}$.  Really we should denote $\Psi_i\colon\R^n\backslash\{0\}\to \R$ with $\Psi_i(x)\coloneqq \psi_i(|x|)$, but we will abuse notation and just use $\psi_i$ in the two senses, with $r$ being interpreted at the radial coordinate.

There is the Laplacian operator $\Delta$ on functions on $\R^n$; on a spherically symmetric function  $\psi(r)$ the Laplacian is given by \[\Delta \psi(r) = \psi''(r) + \tfrac{n-1}{r} \psi'(r).\]
We can now reveal the main property of interest.
\begin{thm} For $i\ge 0$
\[(I- \Delta) \psi_i = (n-1 - 2i)\psi_{i+1}.\]
\end{thm}
\begin{proof}We just use the definition of the Laplacian together with Theorem~\ref{thm:psiDiffEqn} and the inductive definition in Section~\ref{sec:Functions}.
\begin{align*}(I- \Delta) \psi_i(r) &= \psi_i(r) - \psi_i''(r) - \tfrac{n-1}{r} \psi_i'(r)\\
&= \tfrac{2i}{r}\psi'_i(r) - \tfrac{n-1}{r} \psi_i'(r)\\
&= (n-1-2i)\psi_{i+1}(r)
\end{align*}
\end{proof}
In other words, applying the differential operator $I-\Delta$ moves us up the ladder of functions. 
\begin{cor} For $i,k\ge 0$,
\[(I- \Delta)^k \psi_i = \Bigl(\prod_{j=1}^k \bigl(n+1-2(i+j)\bigr)\Bigr)\psi_{i+k}\]
\end{cor}
If $n$ is odd, with $p=(n-1)/2$ then $(I-\Delta)\psi_{p} =0$, so we find that $I-\Delta$ is nilpotent on certain $\psi_i$.
\begin{cor}
If $n=2p+1$ and $0\le i\le p$ then 
\[(I-\Delta)^{p+1} \psi_i=0.\]
\end{cor}
We can now prove the theorem we were aiming for.
\begin{thm}
\label{thm:symsolns}
If $n$ is odd and $g\colon \R^n\{0\} \to \R$ is a spherically symmetric function with $g(r)\to 0$ as $r\to \infty$ satisfying
  \[(I-\Delta)^{(n+1)/2} g=0.\]
then $g$ is a linear combination of $\{\psi_0,\dots,\psi_{p}\}$.
\end{thm}
\begin{proof}
If we define $\overline\psi_0(r):= e^r$ and $\overline\psi_{i+1}(r):=-\tfrac{1}{r}\overline\psi_i(r)$ then you can easily see $\overline\psi_i\in e^{r}\N[\tfrac{1}{r}]$.  All of the arguments used above for the sequence $(\psi_i)$ go through unchanged for $(\overline \psi_i)$, so in particular we find that
$(I-\Delta)^{p+1} \psi_i=0$ for $0\le i\le p$.  This means that the set $\{\psi_0,\dots,\psi_{p},\psi_0,\dots,\psi_{p}\}$ gives us $n-1$ linearly independent solutions to $(I-\Delta)^{p+1} g=0$ which is an order $n-1$ linear ordinary differential equation, so our solutions span the space of solutions.  However, for a solution to decay, as required, it must be a linear combination of the first half of those. 
\end{proof}

%
%

\section{Finding the potential function $h$}
\label{sec:FindingThePotential}
In this section we will find the potential function $h$ of the odd-ball $B^n_R$ in terms of the solution set of a particularly symmetric linear system of equations involving the reverse Bessel polynomials.  

The potential function $h$ of the ball $B^n_R$ will be spherically symmetric, so can be thought of as a radial function $h\colon [0,\infty)\to \R$.  We will first summarize the properties of the potential function.

\begin{thm}[{\cite[Section~3]{BarceloCarbery}}]
The potential function $h\colon [0,\infty)\to \R$ of the odd-ball $B_R^n$, for $n=2p+1$ is such that
\begin{enumerate}[(a)]
\item $h\equiv 1$ on $[0,R]$;
\item \label{item:second}
   $(I-\Delta)^{p+1}h=0$ on $(R,\infty)$;
\item \label{item:third}
  $h(r)\to 0 $ as $r\to \infty$;
\item \label{item:fourth}
$h$ is $p$ times differentiable.
\end{enumerate}
\end{thm}
By conditions (\ref{item:second}) and (\ref{item:third}) with Theorem~\ref{thm:symsolns} above, we obtain the following corollary.

\begin{cor}  There is a set of coefficients $\{\alpha_i\}_i$ (implicitly dependent on $R$) such that the potential function of the $n$-ball $B_R^n$ is of the form
  \[h(r)=\sum_{i=0}^{p} \alpha_i \psi_i(r)\quad\text{for }r>R.\]
\end{cor}
We will use condition (\ref{item:fourth}) at $r=R$ to find these coefficients.  We have $p+1$ unknowns and the final condition essentially gives us $p+1$ constraints at the boundary, so we might hope that these condition determine the potential function uniquely.  That does indeed turn out to be the case.  

Using the fact that $h(r)$ is $p$ times differentiable and is constantly $1$ for $r\le R$ we have the following boundary conditions at $r=R$:
 \begin{equation}
 \label{eqn:naiveboundary}
 h(R)=1,\  h^{(1)}(R)=0,\  h^{(2)}(R)=0,\  \dots,\  h^{(p)}(R)=0.
 \end{equation}
We will need the following lemma, the proof of which is a straightforward induction.
\begin{lemma}
\label{lemma:RecursiveSequence}
If $(g_j(r))_{j=0}^\infty$ is a sequence of functions with $g_{j+1}(r) = -\tfrac{1}{r} g_j'(r)$ then for $j>0$
\[g_j(r)= \sum_{k=1}^j \frac{(-1)^k d_k^j g_0^{(k)}\!(r)}{r^{2j-k}},\]
for non-negative integer constants $d_k^j$ which satisfy $d_k^{j+1}=d_{k-1}^j +(2j-k)d_k^j$, with $d_0^j=0=d_{j+1}^j$ and $d_1^1=1$.  [As $d_0^j=0$ for $j>0$ we can extend the lower limit of the summation to $k=0$.]

In fact, $d_k^j = c_{2j-k}^j$ in the notation of~\cite{BarceloCarbery} and explicitly
\[d_k^j = \frac{(2j-k-1)!}{2^{j-k}(j-k)!\,(k-1)!}\qquad\text{for }1\le k\le j.\]
\qed
\end{lemma}

Now we use the potential function $h(r)$ to define a sequence of functions 
\[
  h_0(r)=\sum_{i=0}^{p} \alpha_i \psi_i(r)
  \quad\text{ and }\quad
  h_{j+1}(r) = -\tfrac{1}{r} h_j'(r),
\] 
so $h_0(r)=h(r)$ for $r\ge R$.  

By Lemma~\ref{lemma:RecursiveSequence} above, the vanishing of the derivatives from $h^{(1)}(R)$ up to  $h^{(j)}(R)$ in~\eqref{eqn:naiveboundary} implies the vanishing $h_j(R)$, so the boundary conditions become
  \begin{equation}
  \label{eqn:lessnaiveboundary}
  h_0(R)=1,\  h_{1}(R)=0,\  h_{2}(R)=0,\  \dots,\  h_{p}(R)=0.
  \end{equation}

However, by using induction and the recursive definition of the function sequence $(\psi_i(r))_{i=1}^{\infty}$, we find 
\[
  h_j(r)=\sum_{i=0}^{p} \alpha_i \psi_{i+j}(r).
\]

Thus we can write these conditions~\eqref{eqn:lessnaiveboundary} as the following system of linear equations.
\[
\begin{pmatrix}
\psi_0(R)&  \psi_1(R)& \psi_2(R)&\dots&\psi_{p}(R)\\
\psi_1(R)& \psi_2(R)& \psi_3(R)& \dots&\psi_{p+1}(R)\\
\psi_2(R)& \psi_3(R)& \psi_4(R)&\dots&\psi_{p+2}(R)\\
\vdots&\vdots&\vdots&&\vdots\\
\psi_{p}(R)&\psi_{p+1}(R)&\psi_{p+2}(R)&\dots&\psi_{2p}(R)
\end{pmatrix}
\begin{pmatrix}
\alpha_0\\
\alpha_1\\
\alpha_2\\
\vdots\\
\alpha_{p}
\end{pmatrix}
=
\begin{pmatrix}
1\\
0\\
0\\
\vdots\\
0
\end{pmatrix}
\] 
The matrix on the left has constant anti-diagonals, and a matrix of this form is known as a \define{Hankel matrix}.
This system is somewhat different to that of Barcel\'o and Carbery, offering an alternative approach.  It is structurally simpler, but involves more terms.  One advantage of this approach will be shown later with the conjectural closed form for the magnitude. 

By rescaling each $\alpha_i$ we can make the entries in the Hankel matrix into reverse Bessel polynomials, that is to say we can get rid of exponentials and negative powers of $r$.  Write $\tilde \alpha_i:= e^{-R}R^{-2i}\alpha_i$ and write $\chi_i= e^R R^{2i}\psi_i$ for the reverse Bessel polynomials then the above system, $\sum_{i=0}^{p} \psi_{i+j}(R)\alpha_i = \delta_{0,j}$ becomes $R^{-2j}\sum_{i=0}^{p} \chi_{i+j}(R)\tilde\alpha_i = \delta_{0,j}$, i.e.
  \[
  \sum_{i=0}^{p} \chi_{i+j}(R)\tilde\alpha_i = \delta_{0,j} \quad\text{for }j=0,\dots,p, 
  \]
Summarizing this all in matrix form we have the following.
\begin{thm}
\label{thm:HankelLinearSystem}
The function
\[
  h(r)=\begin{cases}1& r\in [0,R)\\
    \sum_{i=0}^{p}e^{R-r}(R/r)^{2i} \tilde\alpha_i \chi_{i}(r)& r\in[R,\infty)
    \end{cases}
\]
is the potential function on the ball $B^n_R$ if the sequence $(\tilde \alpha_i(r))_{i=0}^{p}$ is a solution of the following linear system:
\[
\begin{pmatrix}
\chi_0(R)&  \chi_1(R)&\dots&\chi_{p}(R)\\
\chi_1(R)& \chi_2(R)& \dots&\chi_{p+1}(R)\\
\vdots&\vdots&&\vdots\\
\chi_{p}(R)&\chi_{p+1}(R)&\dots&\chi_{2p}(R) \\
\end{pmatrix}
\begin{pmatrix}
\tilde\alpha_0\\
\tilde\alpha_1\\
\vdots\\
\tilde \alpha_{p}
\end{pmatrix}
=
\begin{pmatrix}
1\\
0\\
\vdots\\
0
\end{pmatrix}
.
\]
\qed 
\end{thm}
\noindent Note that $\tilde\alpha_i$ will be a rational function of $R$, for $i=0,\dots,p$.

\section{Calculating the magnitude}
In this section we will get a couple of expressions for the magnitude of odd balls by using the expression in the last section.  First, we use an expression for the magnitude in terms of an integral of the potential function due to Leinster and Meckes.  This gives a linear expression for the magnitude in terms of the solution set to the linear system in Theorem~\ref{thm:HankelLinearSystem}.  We can add this linear expression to the linear system, thus giving the magnitude as an unknown in a linear system: an application of Cramer's Rule gives an formula for the magnitude as a ratio of determinants.  We then use this to calculate some examples.
\subsection{Magnitude formula using the Leinster-Meckes expression}

The goal of this subsection is prove the following theorem.
\begin{thm} 
\label{thm:MagExplicit}
For $n=2p+1$, 
 we have the following expression for the magnitude of the $n$-dimensional ball:
\begin{align*}
  \left|B^n_R\right|
  &= 
  \frac{1}{n!}\left\{R^n +n\sum_{i=0}^{p} \tilde\alpha_i   
  \sum_{j=0}^{p-i} \frac{2^j (p-i)!}{(p-i-j)!}R^{2(p-j)-1}\chi_{i+j+1}(R)\right\},
  \end{align*}  
where $\{\tilde\alpha_i\}_{i=0}^p$ is the set of solutions to the linear system in Theorem~\ref{thm:HankelLinearSystem}.
\end{thm}
The idea is to combine the expression for the potential function $h$ in Theorem~\ref{thm:HankelLinearSystem} with the following theorem of Leinster and Meckes.
\begin{thm}[Leinster-Meckes~\cite{LeinsterMeckes:Survey}]
\label{thm:LeinsterMeckes}
Let $K \subset \R^n$ be a convex body with $n$ odd. If the potential function $h$ is
integrable then the magnitude of $K$ can be obtained by integrating the potential function $h$:
\[|K| = \frac{1}{n!\,\omega_n}\int_{\R^n} h(x)\,\mathrm{d}x.\]
\qed
\end{thm}

To prove Theorem~\ref{thm:MagExplicit} we will need a lemma.
\begin{lemma}
\label{lemma:integral}
For $i$ and $b$ non-negative integers,  $R>0$ 
, we have
  \[
  \int_R^\infty e^{-r}\chi_i(r) r^{2b} \,\mathrm{d}r 
  = 
  e^{-R}\sum_{j=0}^b \frac{2^j b!}{(b-j)!} R^{2(b-j)-1} \chi_{i+j+1}(R).
  \]
\end{lemma}
\begin{proof}
  We proceed by induction on $b$.  
  
  Observe first that
\begin{align*}
  \frac{\dd}{\dd r}\left(-\frac{e^{-r}\chi_{i+1}(r)}{r}\right)
  &= 
  \frac{\dd}{\dd r}\left(-\psi_{i+1}(r)r^{2i+1}\right)\\
  &=
  r\psi_{i+2}(r)r^{2i+1}-\psi_{i+1}(r)(2i+1)r^{2i}\\
  &=
  r^{2i}\left(r^2\psi_{i+2}(r)-(2i+1)\psi_{i+1}(r)\right)\\
  &=
  r^{2i}\psi_{i}(r)\\
  &=
  e^{-r}\chi_i(r).
\end{align*} 
So, by the Fundamental Theorem of Calculus,
\[
  \int_{r=R}^\infty e^{-r} \chi_i(r)\,\dd r 
  =
  \left[-\frac{e^{-r}\chi_{i+1}(r)}{r}\right]^\infty_{r=R}
  =
  \tfrac{1}{R}e^{-R}\chi_{i+1}(R),
  \]
as required.  Thus the statement is true for all $i$ with $b=0$.  
  
  Now assume that it is true for $0\le b< c$;  we will prove it when $b=c$.  Start by using integration by parts together with the above lemma.
  \begin{align*}
    \int_R^\infty e^{-r}\chi_i(r) r^{2c} \,\mathrm{d}r 
    &=
    \left[-\frac{e^{-r}\chi_{i+1}(r)}{r} r^{2c}\right]_R^\infty - 
    \int_R^\infty -\frac{e^{-r}\chi_{i+1}(r)}{r} 2c r^{2c-1}\,\dd r\\
    &= 
    e^{-R}\chi_{i+1}(R)R^{2c-1} + 2c\int_R^\infty e^{-r}\chi_{i+1}(r) r^{2(c-1)}\,\dd r\\
    &=
    e^{-R}\chi_{i+1}(R)R^{2c-1} \\
    &{\quad}\quad
    + 2c   e^{-R}\sum_{j=0}^{c-1}  
       \tfrac{2^j(c-1)!}{(c-1-j)!} R^{2(c-1-j)-1} \chi_{i+1+j+1}(R)\\
    &=
    e^{-R}\chi_{i+1}(R)R^{2c-1} 
     +   e^{-R}\sum_{k=1}^{c}\tfrac{ 2^k c!}{(c-k)!} R^{2(c-k)-1} \chi_{i+k+1}(R)\\
    &=
      e^{-R}\sum_{k=0}^c  \tfrac{2^k c!}{(c-k!)}  R^{2(c-k)-1} \chi_{i+k+1}(R) ,
  \end{align*}
as required, where the third equality used the inductive hypothesis and the fourth equality used the substitution $k=j+1$.  The lemma follows by induction.
\end{proof}

\begin{proof}[Proof of Theorem~\ref{thm:MagExplicit}]
We use Theorem~\ref{thm:LeinsterMeckes} with the expression for the potential function given in Theorem~\ref{thm:HankelLinearSystem}.
\begin{align*}
  |B^n_R| 
  &= 
  \frac{1}{n!\,\omega_n}\int_{\R^n} h(x)\,\mathrm{d}x
  =
  \frac{1}{n!\,\omega_n}\int_{B^n_R} 1\,\mathrm{d}x
  +\frac{1}{n!\,\omega_n}\int_{|x|>R} h(x)\,\mathrm{d}x\\
  &=
  \frac{1}{n!\,\omega_n}R^n \omega_n +\frac{1}{n!\,\omega_n}\int_{r>R} h(r)r^{n-1}\sigma_{n-1}\,\mathrm{d}r\\
  &=
  \frac{1}{n!}\left\{R^n +n\int_{r=R}^\infty h(r)r^{n-1}\,\mathrm{d}x\right\}\\
  &=
  \frac{1}{n!}\left\{R^n +n\sum_{i=0}^{p} \tilde\alpha_i e^R R^{2i}\int_{r=R}^\infty e^{-r}\chi_i(r)r^{-2i}r^{2p}\,\mathrm{d}x\right\} 
\end{align*}
Using Lemma~\ref{lemma:integral} above then gives us the theorem.
\end{proof}

\subsection{Determinant formula for the magnitude}
Now we will give a reasonably explicit formula for the magnitude of an odd dimensional ball in terms of a ratio of two determinants.  

Defining, for $i=0,1,2,\dots,p$, the integral polynomial 
\[\tilde \xi_{p,i} := n \sum_{j=0}^{p-i} \frac{2^j(p-i)!}{(p-i-j)!}R^{2(p-j)}\chi_{i+j+1}(R),\]
recalling $n=2p+1$, and rearranging the formula for magnitude in Theorem~\ref{thm:MagExplicit} we find
\[
  -n!\,R  \left|B^n_R\right| + \sum_{i=0}^{p} \tilde \xi_{p,i} \tilde\alpha_i 
  =
  -R^{2p+2}.
\]
We can extend the linear system in Theorem~\ref{thm:HankelLinearSystem} to deduce that the magnitude $\left|B^n_R\right|$ is obtained by solving the following linear system.
\[
\begin{pmatrix}
\chi_0(R)&  \chi_1(R)&\dots&\chi_{p}(R)&0\\
\chi_1(R)& \chi_2(R)& \dots&\chi_{p+1}(R)&0\\
\vdots&\vdots&&\vdots&\vdots\\
\chi_{p}(R)&\chi_{p+1}(R)&\dots&\chi_{2p}(R)&0 \\
\tilde \xi_{p,0}&\tilde \xi_{p,1}&\dots&\tilde \xi_{p,p}& -n!\, R
\end{pmatrix}
\begin{pmatrix}
\tilde\alpha_0\\
\tilde\alpha_1\\
\vdots\\
\tilde \alpha_{p}\\
|B^n_R|
\end{pmatrix}
=
\begin{pmatrix}
1\\
0\\
\vdots\\
0\\
-R^{2p+2}
\end{pmatrix}
\] 
Clearly this can be simplified by adding a multiple of the top row to the bottom row, so defining, for $i=0,1,\dots, p$ the integral polynomial
  \[
  \xi_{p,i}(R) :=  R^{2p+2}\chi_i(R) + \tilde\xi_{p,i}(R) 
  \]
we find that the magnitude is obtained in the solution to the following linear system.
\[
\begin{pmatrix}
\chi_0(R)&  \chi_1(R)&\dots&\chi_{p}(R)&0\\
\chi_1(R)& \chi_2(R)& \dots&\chi_{p+1}(R)&0\\
\vdots&\vdots&&\vdots&\vdots\\
\chi_{p}(R)&\chi_{p+1}(R)&\dots&\chi_{2p}(R)&0 \\
\xi_{p,0}(R)& \xi_{p,1}(R)&\dots& \xi_{p,p}(R)& -n!\, R
\end{pmatrix}
\begin{pmatrix}
\tilde\alpha_0\\
\tilde\alpha_1\\
\vdots\\
\tilde \alpha_{p}\\
|B^n_R|
\end{pmatrix}
=
\begin{pmatrix}
1\\
0\\
\vdots\\
0\\
0
\end{pmatrix}
\] 
A straightforward application of Cramer's Rule then gives the following reasonably explicit form for the magnitude.
\begin{thm}
\label{thm:MagDet}
  The magnitude $|B^{2p+1}_R|$ of the $2p+1$-dimensional, radius $R$ ball, is given by the following expression involving a ratio of determinants.
  \[
    \left|B^{2p+1}_R\right|
    =
    \frac{(-1)^p}{n!\,R}
    \frac{ \left|
    \begin{matrix}
      \chi_1(R)&  \dots&\chi_{p+1}(R)\\
      \vdots&&\vdots\\
      \chi_{p}(R)&\dots&\chi_{2p}(R) \\
      \xi_{p,0}(R)& \dots& \xi_{p,p}(R)   
    \end{matrix}
    \right|}
    {\left|
    \begin{matrix}
      \chi_0(R)&\dots&\chi_{p}(R)\\
      \vdots&&\vdots\\
      \chi_{p}(R)&\dots&\chi_{2p}(R)
    \end{matrix}
    \right|
    }
  \]\qed
\end{thm}

\subsection{Examples}
It is easy to implement this formula in a dozen or so lines of SageMath.  So for example we have when $n=3$, i.e.~$p=1$,
\begin{align*}
    \left|B^{3}_R\right|
    &=
    \frac{-1}{3!\,R}
    \frac{ \left|
    \begin{matrix}
      \chi_{1}(R) & \chi_{2}(R) \\
      R^{4} \chi_{0}(R) + 3  R^{2} \chi_{1}(R) + 6  \chi_{2}(R) 
      & 
      R^{4} \chi_{1}(R) + 3 R^{2} \chi_{2}(R)
  \end{matrix}\right|}
  {\left|
    \begin{matrix}
      \chi_{0}(R) & \chi_{1}(R) \\
      \chi_{1}(R)& \chi_{2}(R)
  \end{matrix}\right|}\\
  &=
    \frac{-1}{3!\,R}
    \frac{ \left|
    \begin{matrix}
      R & R^{2} + R \\
      R^{4} + 3 R^{3} + 6 R^{2} + 6 R 
      & R^{5} + 3 R^{4} + 3 R^{3}
  \end{matrix}\right|}
  {\left|
    \begin{matrix}
      1 & R \\
      R & R^{2} + R
  \end{matrix}\right|}\\
  &=
    \frac{-1}{3!\,R}
    \frac{- R^{2} (R^{3} + 6 R^{2} + 12 R + 6)}{R}\\
   &=
    \frac{1}{3!}
    (R^{3} + 6 R^{2} + 12 R + 6), 
  \end{align*}
which is exactly as was calculated by Barcelo and Carbery~\cite{BarceloCarbery}.

Similarly, when $n=5$, i.e.~$p=2$ we have
\begin{align*}
    \left|B^{5}_R\right|
    &=
    \frac{1}{5!\,R}
    \frac{
      \left|
    \begin{smallmatrix}
      \chi_{1}(R) & \chi_{2}(R) & \chi_{3}(R) \\
\chi_{2}(R) & \chi_{3}(R) & \chi_{4}(R) \\
R^{6} \chi_{0} + 5 \, R^{4} \chi_{1} + 20 \, R^{2} \chi_{2} + 40 \,
\chi_{3} & R^{6} \chi_{1} + 5 \, R^{4} \chi_{2} + 10 \, R^{2}
\chi_{3} & R^{6} \chi_{2} + 5 \, R^{4} \chi_{3}
  \end{smallmatrix}\right|}
  {\left|
    \begin{matrix}
      \chi_{0}(R) & \chi_{1}(R) &\chi_2(R)\\
      \chi_{1}(R)& \chi_{2}(R)&\chi_3(R)\\
      \chi_{2}(R)& \chi_{3}(R)&\chi_4(R)\\
  \end{matrix}\right|}\\
  &=
  \frac{1}{5!\,R}
    \frac{2 R^{3}  (R^{6} + 18 R^{5} + 135 R^{4} + 525
R^{3} + 1080 R^{2} + 1080 R + 360)}{2R^{2}(R + 3) }\\
&=
  \frac{1}{5!}
    \frac{R^{6} + 18 R^{5} + 135 R^{4} + 525
R^{3} + 1080 R^{2} + 1080 R + 360}{R + 3},
  \end{align*}
which again agrees with the earlier calculation.

\section{Conjecturing the Hankel determinant expression for the magnitude}
In this section we want to see how some empirical (and unexplained) observations about the solution set to the linear system in Theorem~\ref{thm:HankelLinearSystem} lead to a conjecture about a much more symmetric expression for the magnitude of an odd-ball.  This conjecture is proved in \cite{Willerton:Hankel}, but no explanation is given there for the how that expression was guessed.  This is the missing explanation.

First we must embellish the notation a little by indicating on the coefficients of $h$ which dimension we are working in.  So write $\alpha^{[n]}_i$ for the $i$th coefficient of $h$ when we are working with the $n$-dimensional ball $B^{n}_R$. Here are some values of $\tilde\alpha_0^{[n]}$ for small, odd $n$.
\begin{align*}
\tilde\alpha_0^{[1]} &= 1\\
\tilde\alpha_0^{[3]} &=  R + 1 \\
\tilde\alpha_0^{[5]} &= \frac{ R^{3} + 6R^{2} + 12R + 6 }{ 2 !\,( R + 3 )}\\
\tilde\alpha_0^{[7]} &= \frac{ R^{6} + 18R^{5} + 135R^{4} + 525R^{3} + 1080R^{2} + 1080R + 360 }{ 3 !\,( R^{3} + 12 R^{2} + 48 R + 60 )}\\
\tilde\alpha_0^{[9]} &= \frac{ R^{10} + 40R^{9} + 720R^{8} + \dots
+ 1814400R^{2} + 1209600R + 302400 }{ 4 !(\, R^{6} + 30 R^{5} + 375 R^{4} + 2475 R^{3} 
+ 9000 R^{2} + 16920 R + 12600 )}
\end{align*}
You can compare these values with the magnitudes listed in the introduction and arrive at the following which  was also observed independently by Barcelo and Carbery~\cite{BarceloCarbery}.
\begin{observation}
For those $n$ for which the magnitude was calculated, the magnitude of the $n$-ball is a rational function whose numerator is --- up to powers of $R$ --- the same as the zeroth coefficient from the next dimension up:
\[
|B_R^n| \propto \frac{\num\left(\tilde\alpha_0^{[n+2]}\right)}{\den\left(\tilde\alpha_0^{[n]}\right)}.
\]
\end{observation}
Due to the simple form of the right hand side of the linear system in Theorem~\ref{thm:HankelLinearSystem}, Cramer's Rule immediately gives us the following easy to write down expression for $\tilde\alpha_0^{[n]}$.
\[
\tilde\alpha_0^{[n]}
  =
  \frac{\det
  [\chi_{i+j}(R)]_{i,j=1}^{p}
  }
  {\det
  [\chi_{i+j}(R)]_{i,j=0}^{p}
  }.
\]
Using these observations together with some computations to fix the constants of proportionality, one is led to conjecture the following.
\begin{formula}
\label{conj:HankelFormula}
For $n=2p+1$ the magnitude of the $n$-dimensional ball of radius~$R$ has the following form:
  \[|B_R^n|=
  \frac{
\left|  \begin{matrix}
\chi_2(R)&  \chi_3(R)&\dots&\chi_{p+2}(R)\\
\chi_3(R)& \chi_4(R)& \dots&\chi_{p+3}(R)\\
\vdots&\vdots&&\vdots\\
\chi_{p+2}(R)&\chi_{p+3}(R)&\dots&\chi_{2p+2}(R) \\
\end{matrix}
\right|
  }
  {n!\,R\,
 \left|
  \begin{matrix}
\chi_0(R)&  \chi_1(R)&\dots&\chi_{p}(R)\\
\chi_1(R)& \chi_2(R)& \dots&\chi_{p+1}(R)\\
\vdots&\vdots&&\vdots\\
\chi_{p}(R)&\chi_{p+1}(R)&\dots&\chi_{2p}(R) \\
\end{matrix}
 \right|
  }.
\]
\end{formula}
In the presence of Theorem~\ref{thm:MagDet} proving the formula is equivalent to proving the following determinantal identity:
\[
(-1)^p\left|
    \begin{matrix}
      \chi_1(R)&  \dots&\chi_{p+1}(R)\\
      \vdots&&\vdots\\
      \chi_{p}(R)&\dots&\chi_{2p}(R) \\
      \xi_{p,0}(R)& \dots& \xi_{p,p}(R)   
    \end{matrix}
    \right|
    =
    \left|
    \begin{matrix}
      \chi_2(R)&  \dots&\chi_{p+2}(R)\\
      \vdots&&\vdots\\
      \chi_{p+1}(R)&\dots&\chi_{2p+1}(R) \\
      \chi_{2p+2}(R)& \dots& \chi_{2p+2}(R)   
    \end{matrix}
    \right|.
\]
Unfortunately I do not know how to give a direct proof of this identity.   Fortunately, I do know how to prove the formula using a rather different approach, using the weight distribution for the ball rather than the potential function. This is done in~\cite{Willerton:Hankel}, however, using that approach it is not clear why one might guess such a formula for the magnitude.

\section{Conjecturing the derivative of the magnitude}
In the previous section we saw how some observations about the lowest coefficient $\tilde \alpha_ 0 ^{[ n +2]}$ lead to a conjecture about expressing the magnitude of $B^n_R$ in terms of Hankel determinants.  In this final section we see how some observations about the highest coefficient $\tilde \alpha_ {p+1} ^{[ n +2]}$ leads to a conjecture about expressing the \emph{derivative} of the magnitude of $B^n_R$ in terms of Hankel determinants.

We can calculate $\tilde\alpha_p^{[n]}$ for some for small $n$.
\begin{align*}
\tilde\alpha_ 0 ^{[ 1 ]} &= 1\\
\tilde\alpha_ 1 ^{[ 3 ]} &=  -1 \\
\tilde\alpha_ 2 ^{[ 5 ]} &= \frac{ R + 2 }{ 2 !\,( R + 3 )}\\
\tilde\alpha_ 3 ^{[ 7 ]} &= \frac{ -(R^{3} +9R^{2} + 27R + 24) }{ 3 !\,( R^{3} + 12 R^{2} + 48 R + 60 )}\\
\tilde\alpha_ 4 ^{[ 9 ]} &= \frac{ R^{6} + 24R^{5} + 240R^{4} + 1260R^{3} + 3600R^{2} + 5220R + 2880 }{ 4 !\,( R^{6} + 30 R^{5} + 375 R^{4} + 2475 R^{3} + 9000 R^{2} + 16920 R + 12600 )}
\end{align*}

In a seemingly unrelated direction, we know that $\dr{i}h^{[n]}(R)=0$ for $i=1,\dots,p$, and the derivative $\dr{p+1}h^{[n]}$ is discontinuous at $r=R$.  The limit from below at this discontinuity will be $0$ as $h^{[n]}(r)=1$ for $r\le R$ and we can calculate the limit from above for small values of $n$.
\begin{align*}
\lim_{r\downarrow R}\frac{\mathrm{d}  h^{[ 1 ]}}{\mathrm{d}r } &=-1\\
\lim_{r\downarrow R}\frac{\mathrm{d}^ 2  h^{[ 3 ]}}{\mathrm{d}r^ 2 } &= \frac{-( R + 2 )}{R}\\
\lim_{r\downarrow R}\frac{\mathrm{d}^ 3  h^{[ 5 ]}}{\mathrm{d}r^ 3 } &= \frac{-( R^{3} + 9R^{2} + 27R + 24 )}{R^ 2 \,( R + 3 )}\\
\lim_{r\downarrow R}\frac{\mathrm{d}^ 4  h^{[ 7 ]}}{\mathrm{d}r^ 4 } &= \frac{-( R^{6} + 24R^{5} + 240R^{4} + 1260R^{3} + 3600R^{2} + 5220R + 2880 )}{R^ 3 \,( R^{3} + 12R^{2} + 48R + 60 )}
\end{align*}
Again, we see the phenomenon that denominator of something in dimension~$n$ is essentially the denominator of one of the coefficients in dimension~$n+2$.  I have no explanation for this.  But it does lead to some explicit formulas.
  As before due to the simple form of the right hand side of the linear system in Theorem~\ref{thm:HankelLinearSystem}, Cramer's Rule immediately gives us the following easy to write down expression for $\tilde\alpha_0^{[n]}$.
\[\tilde\alpha_p^{[n]}=(-1)^p\frac{\det
[\chi_{i+j+1}(R)]_{i,j=0}^{p-1}
}
{\det
[\chi_{i+j}(R)]_{i,j=0}^{p}
}.\]
With that in mind one can then formulate and prove the following theorem.
\begin{thm}For $n=2p+1$ we have the following expression for the first non-trivial derivative of the potential function at the boundary of the ball:
\[\lim_{r\downarrow R}\dr{p+1}h^{[n]}=
-\frac{\det
[\chi_{i+j+1}(R)]_{i,j=0}^{p}
}
{R^{p+1}\det
[\chi_{i+j}(R)]_{i,j=0}^{p}
}.\]
\end{thm}
\begin{proof}
Take $h_0(r):=\sum_{i=0}^{p}\alpha^{[n]}_i \psi_i(r)$ so that 
\[\lim_{r\downarrow R}\dr{p+1}h^{[n]}=\dr{p+1}h_0(R).\]  Define the sequence $(h_i)_{i=0}^\infty$ as above.
By Lemma~\ref{lemma:RecursiveSequence}, the fact that $\dr{k}h_0(R)=0$ for $k=0, \dots, p$ and the fact that $d_{p+1}^{p+1}=1$ we have 
\[
  h_{p+1}(R)=
  \sum_{k=1}^{p+1}\frac{(-1)^k d_k^{p+1} \dr{k}h_0(R)}{R^{2(p+1)-k}}
=
  \frac{(-1)^{p+1} \dr{p+1}h_0(R)}{R^{p+1}}.
\]
On the other hand, \[h_{p+1}(R) = \sum_{i=0}^{p}\alpha^{[n]}_i \psi_{p+1+i}(R)=R^{-2p-2}
\sum_{i=0}^{p}\tilde\alpha^{[n]}_i \chi_{p+1+i}(R).\]
Equating these two expression gives
\[
  (-1)^{p+1} R^{p+1}\dr{p+1}h_0(R) = \sum_{i=0}^{p}\tilde\alpha^{[n]}_i \chi_{p+1+i}(R).\]
Moving the left hand side to the right hand side and appending the resulting equation to the linear system in Theorem~\ref{thm:HankelLinearSystem} gives
\[
\begin{pmatrix}
\chi_0(R)&  \chi_1(R)&\dots&\chi_{p}(R)&0\\
\chi_1(R)& \chi_2(R)& \dots&\chi_{p+1}(R)&0\\
\vdots&\vdots&&\vdots\\
\chi_{p}(R)&\chi_{p+1}(R)&\dots&\chi_{2p}(R)&0 \\
\chi_{p+1}(R)&\chi_{p+2}(R)&\dots&\chi_{2p+1}(R)&(-1)^p R^{p+1} \\
\end{pmatrix}
\begin{pmatrix}
\tilde\alpha_0\\
\tilde\alpha_1\\
\vdots\\
\tilde \alpha_{p}\\
\dr{p+1}h_0(R)
\end{pmatrix}
=
\begin{pmatrix}
1\\
0\\
\vdots\\
0\\
0
\end{pmatrix}
.
\]

\end{proof}

It turns out that on further experimentation these polynomials crop up in the derivatives of the magnitude function for odd balls.  We can calculate the following derivatives, e.g.~using SageMath.
\begin{align*}
\frac{\mathrm{d}  |B^1_R|}{\mathrm{d}R} &=1\\
\frac{\mathrm{d}|B^3_R|}{\mathrm{d}R } &= \frac{1}{2!}( R + 2 )^2\\
\frac{\mathrm{d}|B^5_R|}{\mathrm{d}R } &= \frac{( R^{3} + 9R^{2} + 27R + 24 )^2}{4!\,( R + 3 )^2}\\
\frac{\mathrm{d} |B^7_R|}{\mathrm{d}R } &= \frac{( R^{6} + 24R^{5} + 240R^{4} + 1260R^{3} + 3600R^{2} + 5220R + 2880 )^2}{6!\,( R^{3} + 12R^{2} + 48R + 60 )^2}
\end{align*}
This leads to the following intriguing conjecture which has been verified numerically up to $n=57$.
\begin{conj}
The derivative of the magnitude of the n-ball is related to the first
non-trivial derivative of the potential function at $r=R$.
\[\frac{\dd}{\dd R} |B_R^n| =
\frac{\bigl(\det
[\chi_{i+j+1}(R)]_{i,j=0}^{p}\bigr)^2
}
{(2p)!\,R^{2}\bigl(\det
[\chi_{i+j}(R)]_{i,j=0}^{p}\bigr)^2
}
=
 \frac{R^{n-1}}{(n-1)!} \left[\lim_{r \downarrow R} \frac{\dd^{p+1}}{\dd r^{p+1}} h(r)\right]^2.\]
\end{conj}
%
%
%
%

{}
\enlargethispage{2em}

\bibliographystyle{acm}
\bibliography{odd_balls}
\end{document}